\documentclass[12pt, reqno]{amsart}
\usepackage{amsmath, amsthm, amscd, amsfonts, amssymb, graphicx, color}
\usepackage[bookmarksnumbered, colorlinks, plainpages]{hyperref}
\hypersetup{colorlinks=true,linkcolor=red, anchorcolor=green, citecolor=cyan, urlcolor=red, filecolor=magenta, pdftoolbar=true}

\newtheorem{theorem}{Theorem}[section]
\newtheorem{lemma}[theorem]{Lemma}
\newtheorem{proposition}[theorem]{Proposition}

\theoremstyle{definition}

\theoremstyle{remark}
\newtheorem{remark}[theorem]{Remark}

\numberwithin{equation}{section}

\def\b1{\bold 1}

\def\tf{\widetilde f}

\begin{document}

\setcounter{page}{1}

\title[Strongly Operator Convex Functions]{Some Methods for Constructing New Operator Monotone Functions from Old Ones}

\author[Lawrence G.~Brown]{Lawrence G.~Brown$^1$}

\address{$^{1}$Department of Mathematics, Purdue University, 150 N.~University Street, West Lafayette, IN\ \ 47907\ USA}
\email{\textcolor[rgb]{0.00,0.00,0.84}{lgb@math.purdue.edu}}






\begin{abstract}
We observe that if $f$ is a continuous function on an interval $I$ and $x_0\in I$, then $f$ is operator monotone if and only if the function $(f(x)-f(x_0))/(x-x_0)$ is strongly operator convex. Then starting with an operator monotone function $f_0$, we construct a strongly operator convex function $f_1$, an (ordinary) operator convex function $f_2$, and then a new operator monotone function $f_3$. The process can be continued to obtain an infinite sequence which cycles between the three classes of functions.  We also describe two other constructions, similar in spirit.  We prove two lemmas which enable a treatment of those aspects of strong operator convexity needed for this paper which is more elementary than previous treatments.  And we discuss the functions $\varphi$ such that the composite $\varphi\circ f$ is operator convex or strongly operator convex whenever $f$ is strongly operator convex.
\end{abstract} \maketitle

\bigskip
We are concerned with three classes of continuous real--valued functions on a non-degenerate interval $I$. The definition and main properties of operator monotone functions are due to Loewner \cite{L}, the main properties of operator convex functions are due to Bendat and Sherman \cite{BS}, and the theory of strongly operator convex function was given initially in \cite[Theorem 2.36]{B1} with an elementary treatment provided later in \cite{B2}.

We have just recently discovered a treatment of much of the theory of strongly operator convex functions even more elementary than the treatment in \cite{B2}.  The main result of \cite{B2}, Theorem 3.2, proves the equivalence of five conditions on a continuous function $f$ defined on an interval $I$ containing $0$.  For a function $f$ on a general interval $I$, $f$ is strongly operator convex if $f(\cdot +x_0)$ satisfies these conditions for some $x_0$ in $I$ (any such $x_0$ will do).  The requirement that $0$ be in $I$ is not very important and is discussed in \cite[Remark 3.3(i)]{B2}.  Of the five equivalent criteria for strong operator convexity provided by \cite[Theorem 3.2]{B2} only three are used in this paper:  Condition (ii) of \cite[Theorem 3.2]{B2} corresponds to (4) below, condition (iv) to (10) below, and condition (v) to (7) below.  Now the proof in \cite{B2} that condition (iv) of Theorem 3.2 implies condition (v) is self-contained and uses only elementary classical analysis.  Lemmas 0.1 and 0.2 below make possible very elementary proofs that \cite[Theorem 3.2(ii)]{B2} implies \cite[Theorem 3.2(iv)]{B2} and that \cite[Theorem 3.2(v)]{B2} implies \cite[Theorem 3.2(ii)]{B2}.  In both of these lemmas and in (4) the operators involved may be assumed to be finite matrices.

For all three classes of functions, there are multiple equivalent criteria, but the easiest (and original) definitions are in terms of operator inequalities.

\medskip\noindent
(1)\  The continuous function $f$ on $I$ is operator monotone if and only if  $h_1\le h_2$ implies $f(h_1)\le f(h_2)$ for self--adjoint operators $h_1$ and $h_2$ with spectra in $I$.

\medskip\noindent
(2)\  The continuous function $f$ on $I$ is operator convex if and only if \newline
$f(th_1+(1-t)h_2)\le tf(h_1)+(1-t)f(h_2)$ for $0\le t\le 1$ and $h_1, h_2$ self--adjoint with spectra in $I$.\newline
Davis \cite{D} showed a different criterion for operator convexity:

\medskip\noindent
(3)\  The continuous function $f$ on $I$ is operator convex if and only if $pf(php)p\le pf(h)p$ for a projection $p$ and a self--adjoint operator $h$ with spectrum in $I$.

\medskip\noindent
(4)\  The continuous function $f$ on $I$ is strongly operator convex if and only if $pf(php)p\le f(h)$ for a projection $p$ and a self--adjoint operator $h$ with spectrum in $I$.

It is obvious that the operator inequality in (4) implies the one in (3).

All three classes of functions can be characterized in terms of integrals:

\medskip\noindent
(5)\  (\cite{L}) The (necessarily continuous) function $f$ on $I$ is operator monotone if and only if $f(x)=ax+b+\int(1/(r-x)-1/(r-x_0))d\mu(r)$, where $a\ge 0$, $x_0\in I$, and $\mu$ is a positive measure on $\Bbb R\setminus I$ such that $\int 1/(r^2+1) d\mu(r)<\infty$.

\medskip\noindent
(6)\  (\cite{BS}) The (necessarily continuous) function $f$ on $I$ is operator convex if and only if 
\begin{equation*}
f(x)=ax^2+bx+c+\int{(x-x_0)^2\over(r-x)(r-x_0)^2} d\mu_+(r)+\int{(x-x_0)^2\over(x-r)(x_0-r)^2} d\mu_-(r),
\end{equation*}
where $a\ge 0$, $\mu_+$ is a positive measure on $\{r:\, r>I\}$, $\mu_-$ is a positive measure  on $\{r:\, r<I\}$,  $\int 1/(1+|r|^3) d\mu_\pm (r)<\infty$, and $x_0$ is in the interior of $I$. (The integrands are obtained from $\pm 1/(r-x)$ by subtracting the first degree Taylor polynomials at $x_0$.)

\medskip\noindent
(7)\  (\cite{B1} or \cite{B2}) The (necessarily continuous) function $f$ on $I$ is strongly operator convex if and only if $f(x)=a+\int 1/(r-x) d\mu_+(r)+\int 1/(x-r) d\mu_-(x)$, where $a\ge 0$, $\mu_+$ is a positive measure on $\{r:\, r>I\}$, $\mu_-$ is a positive measure on $\{r:\, r<I\}$, and 
$\int 1/(1+|r|) d\mu_\pm(r)<\infty$.
\medskip

It is an elementary exercise to deduce the following global criterion from (5) and (7):

\medskip\noindent
(8)\  The continuous function $f$ on $I$ is operator monotone if and only if the function $g$ where $g(x)=(f(x)-f(x_0))/(x-x_0)$, is strongly operator convex. Here $x_0$ can be any point in $I$.
\medskip

If $x_0$ is in the interior of $I$, then the domain of $g$ is $I$ (necessarily $f$ is real analytic in the interior of $I$), and if $x_0$ is an endpoint of $I$ the domain of $g$ is $I\setminus \{x_0\}$. The only comment that we will make about the proof is that the measures $\mu_\pm$ in (7) are obtained from the measure $\mu$ in (5) by multiplying by $1/|x_0-r|$.

Of course there is a well known global criterion for operator monotonicity due to Loewner namely, $f$ is operator monotone if and only if the holomorphic extension of $f$ maps the upper half plane into itself.

The following are two more previously known global criteria:

\medskip\noindent
(9)\  (\cite{BS} and \cite[Lemma 2.1]{U}) The continuous function $f$ on $I$ is operator convex if and only if the function $g$, where $g(x)=(f(x)-f(x_0))/(x-x_0)$, is operator monotone. Here  $x_0$ can be any point in $I$.

\medskip\noindent
(10)\  (\cite{B1} or \cite{B2}) The continuous function $f$ on $I$ is strongly operator convex if and only if either $f=0$ or $f(x)>0$, $\forall x$, and $- 1/f$ is operator convex.
\medskip

The following comment on (8) seems somewhat interesting: Suppose $g$ is a strongly operator convex function on $I$ and $b$ is a finite endpoint of $I$, say the right endpoint, which is not in $I$. Then if $f(x)=(x-b)g(x)$, $f$ can be extended to an operator monotone function $\tf$ on $I\cup\{b\}$ and $(\tf(x)-\tf(b))/(x-b)=g(x)-\delta/(b-x)$, where $\delta=\mu_+(\{b\})$. A similar comment applies to (9), and we do not know whether this was already known. (It is an elementary exercise to deduce (9) from (5) and (6), and it is this kind of reasoning which is used to derive the comment in question. Uchiyama's proof of \cite[Lemma 2.1]{U} used this kind of reasoning, but that result was stated only for open intervals.)

\begin{lemma}
Let $I$ be a non-degenerate interval, and let $f$ be a continuous function on $I$ such that $f(x)>0$, $\forall x$.  Then if $pf(php)p \le f(h)$ for all projections $p$ and self-adjoint operators $h$ with spectrum in $I$, it follows that $-1/f$ is operator convex.
\end{lemma}

\begin{proof}We represent operators by $2\times2$ matrices relative to $\b1 = p + (\b1 -p)$ and first establish some well known facts.  Suppose $k=\begin{pmatrix} a &b^*\\ b &c\end{pmatrix}$ where $c$ is positive and invertible.  Then 

$$ k=\begin{pmatrix} \b1 &b^*c^{-1}\\ 0 &\b1\end{pmatrix} \begin{pmatrix} a-b^*c^{-1}b &0\\ 0 &c\end{pmatrix} \begin{pmatrix} \b1 &0\\ c^{-1}b &\b1\end{pmatrix}.$$

It follows that $k\ge 0$ if and only if $a-b^*c^{-1}b \ge 0$ and that $k$ is positive and invertible if and only if $a-b^*c^{-1}b$ is positive and invertible.  If $k$ is invertible, then 

$$ k^{-1}=\begin{pmatrix} \b1 &0\\ -c^{-1}b &\b1\end{pmatrix} \begin{pmatrix} (a-b^*c^{-1}b)^{-1} &0\\ 0 &c^{-1}\end{pmatrix} \begin{pmatrix} \b1 &-b^*c^{-1}\\ 0 &1\end{pmatrix},
$$

whence $k^{-1}=\begin{pmatrix} (a-b^*c^{-1}b)^{-1} &*\\ * &*\end{pmatrix}$.

Now take $p$ and $h$ as above, let $k=f(h)$, and let $pf(php)p = \begin{pmatrix} y &0\\ 0 &0\end{pmatrix}$. We will prove $-1/f$ operator convex by using Davis' condition (3).  Thus we are given $\begin{pmatrix} y &0\\ 0 &0\end{pmatrix} \le \begin{pmatrix} a &b^*\\ b &c\end{pmatrix}$, and we need to prove $-y^{-1} \le -(a-b^*c^{-1}b)^{-1}$.  It is equivalent to prove $y \le a-b^*c^{-1}b$.  But the fact that $\begin{pmatrix} a-y &b^*\\ b &c\end{pmatrix} \ge 0$ implies $a-y-b^*c^{-1}b \ge 0$ as shown above.
\end{proof}

\begin{lemma}Let $f(x)=1/x$ for $x>0$.  Then $pf(php)p \le f(h)$ for all projections $p$ and positive invertible operators $h$.
\end{lemma}

\begin{proof}We represent operators by $2\times 2$ matrices relative to $\b1 = (\b1 -p)+p$.  Thus if $h=\begin{pmatrix} a &b^*\\ b &c\end{pmatrix}$, we need to show 
$$
\begin{pmatrix} 0 &0\\ 0 &c^{-1}\end{pmatrix}\le h^{-1}=\begin{pmatrix} (a-b^*c^{-1}b)^{-1} &-(a-b^*c^{-1}b)^{-1}b^*c^{-1}\\ -c^{-1}b(a-b^*c^{-1}b)^{-1} &c^{-1}+c^{-1}b(a-b^*c^{-1}b)^{-1}b^*c^{-1}\end{pmatrix}.
$$
Thus we need to prove $\begin{pmatrix} (a-b^*c^{-1}b)^{-1} &-(a-b^*c^{-1}b)^{-1}b^*c^{-1}\\ -c^{-1}b(a-b^*c^{-1}b)^{-1} &c^{-1}b(a-b^*c^{-1}b)^{-1}b^*c^{-1}\end{pmatrix} \ge 0$,
and this follows from a symmetric version of the criterion in the previous proof.  We need:
$$c^{-1}b(a-b^*c^{-1}b)^{-1}b^*c^{-1}-[-c^{-1}b(a-b^*c^{-1}b)^{-1}](a-b^*c^{-1}b)[-(a-b^*c^{-1}b)^{-1}b^*c^{-1}]\ge 0,
$$
\noindent and this is obvious.
\end{proof}
We can now proceed with the main construction. Start with a non-constant operator monotone function $f_0$ on $I$. Choose $x_0$ in $I$ and let $f_1(x)=(f_0(x)-f_0(x_0))/(x-x_0)$. Since $f_1$ is a non-zero strongly operator convex function, we can define a (strictly negative) operator convex function $f_2$ by setting $f_2=-1/f_1$. Then choose $x_1$ in the domain of $f_2$ and let $f_3(x)=(f_2(x)-f_2(x_1))/(x-x_1)$, a new operator monotone function.

For example take $f_0(x)=x^\alpha$, $x\ge 0$, where $0<\alpha<1$, $x_0=1$ and $x_1=0$. Then $f_3(x)=(x^{\alpha-1}-1)/(x^\alpha-1)$ on $(0,\infty)$. The direct proof that $f_3$ is operator monotone, by checking that it maps the upper half place into itself, involves a somewhat non-trivial calculus problem. Therefore the fact that our construction can easily produce such a function and guarantee that it is operator monotone may be interesting.

The process can be continued to obtain an infinite sequence of functions such that each $f_{3n}$ is operator monotone, each $f_{3n+1}$ is strongly eperator convex, and each $f_{3n+2}$ is operator convex, provided that no $f_{3n+1}$ is 0. It is not hard to see that this will be the case unless $f_0$ is rational. If $f_0$ is rational, then the degree of $f_{3n+3}$, defined as the maximum of the degrees of the numerator and denominator, will be one less than the degree of $f_{3n}$, and the process will eventually terminate. The construction depends on an infinite sequence $\{x_i\}$ of points used in the transitions from $f_{3n}$ to $f_{3n+1}$ and from $f_{3n+2}$ to $f_{3n+3}$. The requirements for this sequence are a little complicated.  Since a negative convex function always has a finite limit at any finite endpoint of its domain, the functions $f_{3n+2}$ may always be considered to be defined on $\bar{I}$.  Thus $x_i$ for $i$ odd, which is used in the transition from $f_{3n+2}$ to $f_{3n+3}$ is always allowed to be an endpoint of $I$.  But if such an $x_i$ is an endpoint, then in general $x_{i+1}$ cannot be the same endpoint.

Another construction starts with an operator monotone function $f^*_0$. Then the strongly operator convex function $f^*_1$ is defined, as before, by $f^*_1(x)=(f^*_0(x)-f^*_0(x_0))/(x-x_0)$. But now we define a new operator monotone function $f^*_2$ by $f^*_2(x)=(f^*_1(x)-f^*_1(x_1))/(x-x_1)$. We can continue the process to obtain an infinite sequence which alternates between operator monotone functions and strongly operator convex functions. However, the theory of strongly operator convex functions is not really being used, since we need only the fact that $f^*_{2n+1}$ is operator convex. Therefore this construction could have been done a long time ago, and perhaps it has already been done. Also recall that since the measures, as in (5) and (7), for $f^*_{n+1}$ are obtained from those for $f^*_n$ by multiplying by $1/|r-x_n|$, all of the measures will be finite for $n\ge 2$, and they will have increasingly many finite moments as $n$ increases. Thus the operator monotone functions $f^*_{2n}$ will not be ''general'' but will have special properties. Finally, there is one more reason why we think that the $\{f^*_n\}$ process is less interesting than the original one. Suppose $f$ is a strictly positive continuous function on $I$ and $g=-1/f$. Then $g$ operator convex implies $f$ operator convex but the convese is false. Therefore if $h_1(x)=(g(x)-g(x_1))/(x-x_1)$ and $h_2(x)=(f(x)-f(x_1))/(x-x_1)$, then $h_1$ operator monotone implies $h_2$ operator monotone but not conversely. Therefore it is presumably more interesting to know that $h_1$ is operator monotone than to know that $h_2$ is operator monotone. If $f$ is the function $f_1$ or $f^*_1$ in the above constructions, then the first construction produces $h_1$ and the second produces $h_2$.

As an example, we can make the same choices as in the previous example, by setting $f^*_0(x)=x^\alpha$, $x_0=1$, and $x_1=0$. Then $f^*_2(x)=(x^{\alpha-1}-1)/(x-1)$ on $(0,\infty)$. We have claimed that $f^*_2$ is a less interesting example of operator monotonicity than the previous $f_3$, with $f_3(x)=(x^{\alpha-1}-1)/(x^\alpha-1)$. But the direct verification that $f^*_2$ maps the upper half plane into itself does not seem any easier than that for $f_3$. Thus the $\{f^*_n\}$ process may be interesting if it is really new.

It is also possible to run the first process backwards, though this is slightly problematical. Thus start with an operator monotone function $f_0$ on $I$ and let $f_{-1}(x)=f_0(x)(x-x_0)+c_0$. It is necessary to choose $c_0$ small enough that $f_{-1}(x)<0$, $\forall x$, and this is not always possible. But it is possible if both $I$ and $f_0$ are bounded. So if necessary we can restrict $f_0$ to a smaller interval in order to construct a suitable operator convex function $f_{-1}$. Then let $f_{-2}=-1/f_{-1}$, and take $f_{-3}(x)=f_{-2}(x)(x-x_1)+c_1$. So $f_{-3}$ is a new operator monotone function. By the comments after (10) it is permissible for some of the $x_i$'s to be endpoints of $I$ even if those endpoints are not in $I$. But if we want really to be running the original process backwards, it is necessary that the singletons $\{x_i\}$ have measure 0.

As a first example we can take $f_0(x)=x^\alpha$ on $[0,2]$, $0<\alpha<1$, $x_0=1$ and $x_1=0$, similarly to the above examples. Then $f_{-1}(x)$ might be $x^\alpha(x-1)-3$ and $f_{-3}(x)$ might be $x/(3-x^\alpha(x-1))$. The direct proof that $f_{-3}$ maps the upper half plane into itself is trivial. For a second example take $f_0(x)=x^\alpha-(2-x)^\alpha$ on $[0,2]$ and make the same choices for $x_0, x_1,c_0$, and $c_1$. Then $f_{-3}(x)=x/(3-(x-1)(x^\alpha-(2-x)^\alpha)$. A direct verification that $f_{-3}$ maps the upper half plane into itself does not seem trivial.

It might be interesting to investigate the behavior of $\{f_n\}, \{f^*_n\}$, and $\{f_{-n}\}$ as $n\to\infty$. We have not attempted this, since we have no expertise in dynamical systems, but the situation for $\{f^*_n\}$ is probably not very difficult. For $n\ge 2$ the $f^*_n$'s are determined completely by the measures appearing in (5) or (7) (the $1/(r-x_0)$ term in (5) can be omitted when $n\ge 2$), and it is easy to calculate the measures in terms of the measure for $f^*_0$ and the $x_i$'s. If the length of $I$ is 2 and we choose $x_i$ to be the midpoint of $I$ for all $i$, then $\{f^*_{2n}\}$ and $\{f^*_{2n-1}\}$ will have possibly non-zero limits as $n\to\infty$.

It is well known and trivial to prove that if $\varphi$ is both operator convex and operator monotone, then the composite $\varphi\circ f$ is operator convex whenever $f$ is operator convex.  (Note that the natural domain of $\varphi$ is necessarily unbounded to the left.  Thus if this domain contains any positive numbers, it must also contain $0$.) If $f$ is strongly operator convex, then $\varphi$ need not be operator convex.

\begin{proposition}
Let $\varphi$ be an operator monotone function on an interval $I$ and $f$ a strongly operator convex function whose range lies in $I$.

(i) If $0\in I$ and $\varphi(0) \ge 0$, then $\varphi\circ f$ is strongly operator convex.

(ii) If either $0\in I$ or $0$ is the left endpoint of $I$, then $\varphi\circ f$ is operator convex.
\end{proposition}

\begin{proof}
(i)  We use (5) for $\varphi$ with $x_0 = 0$.  Let $\varphi_r(x)=1/(r-x)-1/r=x/r(r-x)$ for $r\notin I$.  Since $b=\varphi(0)\ge0$, it is sufficient to show $\varphi_r\circ f$ is strongly operator convex for each $r$.  By (10) this is equivalent to operator convexity of $-1/\varphi_r\circ f$.  Since $-1/\varphi_r\circ f = r-r^2/f$, this is clear.

(ii)  Again we use (5), but now $x_0$ need not be $0$.  Since operator convexity is preserved by additive constants, it is enough to show $g_r\circ f$ is operator convex for each $r\notin I$, where $g_r(x) =1/(r-x)$.  If $r$ is to the right of $I$, this follows from the fact that $g_r$ is operator monotone and operator convex.  If $r$ is to the left of $I$, then $r\le0$, and the result follows from (10), since $f-r$ is strongly operator convex.
\end{proof}

We do not know a reference for the following proposition, though we think it ought to be already known.

\begin{proposition}
If $\varphi$ is a function on an interval $I$, and if $\varphi\circ f$ is operator convex for every operator convex function $f$ with range in $I$, then $\varphi$ is operator monotone and operator convex.
\end{proposition}

\begin{proof}
Since $f$ can be the identity function on $I$, $\varphi$ must be operator convex (and in particular, continuous).  Since $\varphi$ is operator monotone if and only if its restriction to every bounded open subinterval of $I$ is operator monotone, we may assume $I$ is bounded and open.  If $a,b\in \Bbb R$ and $a>0$, then $f$ is operator convex if and only if $af+b$ is operator convex.  Also $\varphi$ is operator monotone if and only if the function $x\mapsto \varphi(ax+b)$ is operator monotone.  Therefore, after a linear change of variable, we are reduced to the case $I=(-1,1)$.

Now if $g(x)=\varphi(x^2)$, $-1<x<1$, we will deduce the operator monotonicity of $\varphi$ from the operator convexity of $g$.  Let $\tilde{\varphi}$ and $\tilde{g}$ be the holomorphic extensions of $\varphi$ and $g$ to $I\cup\{z\colon \text{Im}z\ne0\}$, and let $g_1(z) =\tilde{\varphi}(z^2)$.  Since the domain of $g_1$ is a connected open subset of the domain of $\tilde{g}$, we see from the principle of uniqueness of analytic continuation that $\tilde{g}$ extends $g_1$.  Thus, $g_1$ continues analytically across the imaginary axis and $\tilde{\varphi}$ continues analytically across the negative real axis.  So if we apply (6) to $\varphi$, we find that $\mu_{-}=0$.

Now applying (6) to $\varphi$ and $g$ with $x_0=0$, we have

\medskip\noindent
(11)\qquad $\varphi(x)=ax^2 + bx +c +\int{(1/(r-x)-1/r-x/r^2)} d\mu_+(r)$, and
 
\medskip\noindent
(12)\qquad $g(x)=\alpha x^2+\beta x+\gamma +\int{x^2\over s^2(s-x)} d\nu_+(s) -\int{x^2\over s^2(s-x)} d\nu_-(s)$.

There is a method for finding the measures $\mu_{\pm}$ for an operator convex function based on the fact that the Poisson kernel for the upper halfplane is given by $(1/\pi)\text{Im}(1/(r-z))$.  If $\tilde{h}$ is the holomorphic extension of an operator convex function $h$, one restricts $(\pm 1/\pi)\text{Im }\tilde{h}$ to lines parallel to the real axis and takes the distributional limit as the lines approach the real axis from above.  Using this and the identity $1/(r-z^2)=(1/2\sqrt r)((1/\sqrt r-z)+1/(\sqrt r+z))$, we see that $d\nu_+=(1/2s)d\mu_+(s^2)$.  Now since $\int{(1/s^3)} d\nu_+(s) <\infty$, we conclude that $\int{(1/s^4)} d\mu_+(s^2) <\infty$ and $\int{(1/r^2)} d\mu_+(r)<\infty$.

Now we can replace (11) by

\medskip\noindent
(11$^{'}$)\qquad $\varphi(x)=ax^2+b^{'}x+c+\int{(1/(r-x)-1/r)} d\mu_+(r)$,

\medskip\noindent
where $b^{'}=b-\int{(1/r^2)} d\mu_+(r)$.

Also, since $g$ is even, $d\nu_-(s)=d\nu_+(-s)$, and (12) simplifies to 

\medskip\noindent
(12$^{'}$)\qquad $g(x)=\alpha x^2 + \gamma + \int{(1/(s-x)+1/(s+x)-2/s)} d\nu_+(s)$.

Finally, by comparing (11$^{'}$) and (12$^{'}$) and recalling that $g(x) =\varphi(x^2)$, we have that $a=0$, $b^{'}=\alpha$, and $c=\gamma$.  Thus, by (5), $\varphi$ is operator monotone.
\end{proof}

\begin{remark}
It follows from the above proof that $f$ can be restricted to be a quadratic function.
\end{remark}

Using Proposition 0.4 and (10), it is easy to see that any function $\varphi$ which satisfies the conclusion of either part of Proposition 0.3 (for all $f$), must have an extension which satisfies the hypothesis.  It is also easy to find the functions $\varphi$ such that $f$ operator convex implies $\varphi\circ f$ strongly operator convex.  These are the operator monotone functions whose natural domain is unbounded to the left and which are non-negative on their natural domain.  An equivalent condition is that $\varphi$ must be both operator monotone and strongly operator convex.

There are other ways to derive new operator monotone functions from old ones.  One possibility is to start with an operator monotone function $f$, define an operator convex function $g$ by $g(x) = (x-x_0)f(x)$, and then get a new operator monotone function $h$ with $h(x) =(g(x)-g(x_1))/(x-x_1)$.  This technique was used by Hansen and Pedersen to derive \cite[Theorem 3.9]{HP}.  One could also use the same technique to derive a new strongly operator convex function $h$ from a given strongly operator convex function $f$.  This step could then be inserted into one of the processes given above.  Or, in the main process, instead of defining $f_{3n+2}=-1/f_{3n+1}$, one could take $f_{3n+2}=\varphi\circ f_{3n+1}$, where $\varphi$ satisfies part (ii) of Proposition 0.3.  Clearly there are so many ways to move from one of the three classes of functions to another (or to the same class) that it would be pointless to try to list them all.  Possibly the three processes that we described explicitly above are among the more interesting ways to do this.

\bibliographystyle{amsplain}

\begin{thebibliography}{99}

\bibitem{BS} J.~Bendat and S.~Sherman, 
\textit{Monotone and convex operator functions}, Trans.~Amer. Math Soc. {\bf 79} (1955), 58--71.

\bibitem{B1} L.G.~Brown, 
\textit{Semicontinuity and multipliers of $C^*$--algebras}, Can.~J.~Math.
\textbf{40} (1988), 865--988.

\bibitem{B2} L.G.~Brown, 
\textit{A treatment of strongly operator convex functions that does not require any knowledge of operator algebras}, Ann.~Funct.~Anal., to appear, arXiv \#  1407.5116

\bibitem{D} C.~Davis,
\textit{A Schwarz inequality for convex operator functions}, Proc.~Amer.~Math.~Soc.,
\textbf{8} (1957), 42--44.

\bibitem{HP} F.~Hansen and G.~K.~Pedersen,
\textit{Jensen's inequality for operators and L\"owner's theorem}, Math.~Ann.,
\textbf{258} (1982), 229--241.

\bibitem{L} K.~Loewner, 
\textit{\"{U}ber monotone Matrixfunktionen}, Math.~Z., 
\textbf{38} (1934), 177--216.

\bibitem{U} M.~Uchiyama,
\textit{Operator monotone functions, positive definite kernels and majorization}, Proc.~Amer.~Math.~Soc.,
\textbf{138} (2010), 3985--3996.

\end{thebibliography}

\end{document}